\documentclass[12pt,reqno]{amsart}
\usepackage{amssymb}
\usepackage{mathtools}
\usepackage[english]{babel}
\usepackage[usenames]{color}
\usepackage{ulem}

\usepackage{amsmath,amsfonts,amssymb}
\usepackage{mathtools}
\usepackage{mathrsfs}
\usepackage[all]{xy}
\usepackage{graphicx}
\usepackage{latexsym}
\usepackage{verbatim}
\numberwithin{equation}{section}

\def\v{\varphi}
\def\Re{{\sf Re}\,}
\def\Im{{\sf Im}\,}

\newcommand{\D}{\mathbb D}

\newcommand{\R}{\mathbb R}
\newcommand{\Ha}{\mathbb H}

\newcommand{\Z}{\mathbb Z}
\newcommand{\C}{\mathbb C}

\newcommand{\B}{\mathbb B}

\newcommand{\N}{\mathbb N}

\def\Re{{\sf Re}\,}
\def\Im{{\sf Im}\,}

\def\Re{{\sf Re}\,}
\def\Im{{\sf Im}\,}

\def\v{\varphi}
\def\Re{{\sf Re}\,}
\def\Im{{\sf Im}\,}

\def\1#1{\overline{#1}}
\def\2#1{\widetilde{#1}}
\def\3#1{\widehat{#1}}
\def\4#1{\mathbb{#1}}
\def\5#1{\frak{#1}}
\def\6#1{{\mathcal{#1}}}

\def\Re{{\sf Re}\,}
\def\Im{{\sf Im}\,}

\newcommand{\mcite}[1]{\csname b@#1\endcsname}

\theoremstyle{theorem}

\def\Re{{\sf Re}\,}
\def\Im{{\sf Im}\,}

\def\Arg{{\rm Arg}}

\newtheorem{theorem}{Theorem}[section]
\newtheorem{lemma}[theorem]{Lemma}
\newtheorem{proposition}[theorem]{Proposition}
\newtheorem{corollary}[theorem]{Corollary}

\theoremstyle{definition}
\newtheorem{definition}[theorem]{Definition}

\theoremstyle{remark}
\newtheorem{remark}[theorem]{Remark}

\newtheorem{question}[theorem]{Question}

\numberwithin{equation}{section}

\title[A counterexample to parabolic dichotomy]{A counterexample to  parabolic dichotomies in holomorphic iteration}

\author{Leandro Arosio$^1$, Filippo Bracci$^1$ \and Herv\'e Gaussier$^2$}

\address{Leandro Arosio,
Dipartimento Di Matematica, Universit\`a di Roma \lq\lq Tor Vergata\rq\rq, Via Della Ricerca Scientifica 1, 00133 Roma, Italy}
\email{arosio@mat.uniroma2.it}
\address{Filippo Bracci, Dipartimento Di Matematica, Universit\`a di Roma \lq\lq Tor Vergata\rq\rq, Via Della Ricerca Scientifica 1, 00133 Roma, Italy}
\email{bracci@mat.uniroma2.it}
\address{Herv\'e Gaussier, Univ. Grenoble Alpes, IF, F-38000 Grenoble, France and CNRS, IF, F-38000 Grenoble, France}
\email{herve.gaussier@univ-grenoble-alpes.fr}


\thanks{$^1\,$Partially supported by  PRIN {\sl  Real and Complex Manifolds: Geometry and Holomorphic Dynamics} n. 2022AP8HZ9, by INdAM, and by the   MUR Excellence Department Project MatMod@TOV
CUP:E83C23000330006}
\thanks{$^2\,$ Partially supported by ERC ALKAGE}

\begin{document}

\begin{abstract}
We give an example of a parabolic holomorphic self-map $f$ of the unit ball $\B^2\subset \C^2$ whose  canonical Kobayashi hyperbolic semi-model is given by an elliptic automorphism of the disc $\D\subset \C$, which can be chosen to be different from the identity. As a consequence, in contrast to the one dimensional case, this provides a first example of a holomorphic self-map of the unit ball which has points with zero hyperbolic step and points with nonzero hyperbolic step, solving an open question and showing that parabolic dynamics in the ball $\B^2$ is radically different from parabolic dynamics in the disc.
The example is obtained via a geometric method, embedding the ball $\B^2$  as a domain $\Omega$ in  the bidisc $\D\times \mathbb{H}$ that  is forward invariant and absorbing for the map $(z,w)\mapsto (e^{i\theta}z,w+1)$, where $\Ha\subset \C$ denotes the right half-plane.  We also show that  a complete Kobayashi hyperbolic domain $\Omega$ with such properties cannot be Gromov hyperbolic  w.r.t. the Kobayashi distance (hence, it cannot be biholomorphic to $\B^2$) if an additional quantitative geometric condition is satisfied.
\end{abstract}
\maketitle
\tableofcontents

\section{Introduction}
\subsection{Two conjectural dichotomies in holomorphic iteration} Let $\B^q$ be the unit ball in $\C^q$, $q\geq 1$, and let $f\colon \B^q\to \B^q$ be a holomorphic self-map. The dynamics of $f$ has been the object of intensive study since the foundational works of Julia, Wolff and Denjoy, both in one variable and, more recently, in several  variables (see  \cite{AbBook, Abbook2} for a monograph on this subject or  \cite{PPC} for a recent survey).
The map $f$ is called {\sl elliptic} if it has a fixed point in $\B^q$. If $f$ is non-elliptic, then by Herv\'e's generalization of the classical  Denjoy--Wolff Theorem there exists a point $\zeta\in \partial \B^q$ such that $f^n(z)\to \zeta$ for all $z\in \B^q$.

A sort of derivative can be calculated at $\zeta$, even if $f$ is not necessarily continuous at $\zeta$. The number $$\lambda_\zeta:=\liminf_{z\to \zeta}\frac{1-\|f(z)\|}{1-\|z\|}$$ satisfies $0<\lambda_\zeta\leq 1$ and is called the {\sl dilation} of $f$ at $\zeta$.
A non-elliptic map $f$ is called {\sl parabolic} if  $\lambda_\zeta=1$, and is called
 {\sl hyperbolic} if   $0<\lambda_\zeta<1$.
For all $z\in \B^q$, we define the {\sl (hyperbolic) step} of $f$ at $z$ by
$$s(z):=\lim_{n\to\infty} k_{\B^q}(f^n(z), f^{n+1}(z))\in [0,+\infty),$$ which is well-defined since $f$ does not expand the  Kobayashi distance $k_{\B^q}$ of the ball.
%The map $f$ is said to be {\sl zero-step} if ${\rm step}_f(z)=0$ for all $z\in \B^q$, and it is said to be  {\sl nonzero-step} if ${\rm step}_f(z)>0$ for all  $z\in \B^q$.
Hyperbolic self-maps have nonzero step for every $z\in \B^q$, while there exist  examples of parabolic maps with  all points of zero step ({\sl e.g.}, $z\mapsto z+1$ in the right half-plane $\mathbb{H}$)
 and examples with all points of nonzero step ({\sl e.g.}, a parabolic automorphism). If $q=1$, then this is actually a dichotomy: if $f$ is parabolic then either every point has zero step or every point has nonzero step (\cite{pommerenke, BP}). In other words being {\sl zero-step} or {\sl nonzero-step} is a property of the map $f$. This fact plays a crucial role in the study of the dynamics of parabolic maps.  The following question was asked in 2011 by Poggi-Corradini in \cite[Section 2.4]{PPC}.
\begin{question}[Parabolic step dichotomy]\label{Q0}
Does this  dichotomy also hold if $q>1$?
\end{question}

This  is strongly  related to an open question in the theory of canonical models of holomorphic self-maps of $\B^q$. In \cite{Arosio, Arosio-Bracci} it is proved that every holomorphic self-map $f$ of $\B^q$ admits a (essentially unique) {\sl canonical Kobayashi hyperbolic semi-model}, (or just {\sl canonical model}, for short) that is, a semi-conjugacy $\ell\circ  f=\tau\circ \ell$, where $\ell\colon \B^q\to \B^d$ is a holomorphic map,  $d\in \{ 0, 1, \ldots, q\}$,  $\tau$ is an automorphism of $\B^d$, $\ell(\B^q)$ is {\sl $\tau$-absorbing} (that is, every orbit of $\tau$ is eventually contained in $\ell(\B^q)$), and the triple $(\B^d,\ell,\tau)$ satisfies a natural ``maximality'' property, {\sl i.e.}, every other semi-conjugacy of $f$ to an automorphism of a  Kobayashi hyperbolic complex manifold factorizes through $\ell$. 

The number $d$ introduced in the canonical model is a dynamical invariant of $f$ and it is called the {\sl type} of $f$ (see \cite{Ar-Br2}). 
When $d=0$ the canonical model  degenerates to a point, and we say that the canonical model is {\sl trivial}.  

There is an interesting relation between the canonical model and the step of the map (see, \cite[Theorem~4.6]{Arosio}): for all
$z\in \B^q$, one has 
\begin{equation}\label{step-modello}
s(z)=k_{\B^d}(\ell(z),\tau(\ell(z))),
\end{equation}
and thus it is easy to see that  $f$ has a point of zero step if and only if the automorphism $\tau$ has a fixed point.

It is proved in \cite{Arosio, Arosio-Bracci} that a hyperbolic  self-map has a hyperbolic canonical model $\tau\colon\B^d\to \B^d$ with $0< d\leq q$, and  the dilation of $f$ and $\tau$ at the respective Denjoy--Wolff points is the same. 
Similarly,  a parabolic  holomorphic self-map with all points of nonzero step has a parabolic canonical model $\tau\colon\B^d\to \B^d$ with $0< d\leq q$
(this generalizes to higher dimension the classical theorems in the disc of Valiron \cite{Valiron} and Pommerenke \cite{pommerenke}).
 
The following question was asked in \cite[p. 3330]{Arosio-Bracci} (see also \cite{Arosio-Cortona} for a short survey about this question).
\begin{question}[Parabolic model dichotomy]\label{Q} 
%Is the canonical semi-model of a parabolic holomorphic self-map of $\B^q$ with zero step necessarily trivial? 
Is the canonical model of a parabolic holomorphic self-map of $\B^q$ either parabolic or trivial?
\end{question}
Since  a parabolic map cannot have a hyperbolic canonical model (see e.g. \cite{Arosio-Bracci}), it follows that a   parabolic holomorphic self-map of $\B^q$ contradicts parabolic model dichotomy if and only if it has
a nontrivial elliptic canonical model.
 If $q=1$, then it follows from the classical linearization results in one variable (see  Baker and Pommerenke \cite{BP,pommerenke} and Cowen \cite{Cowen}) that parabolic  model dichotomy holds.
   Also,  it follows from results in \cite{bayart} that if $f$ is a fractional linear univalent self-map of $\B^q$, then parabolic model dichotomy  holds (see Example 5.19 in \cite{Arosio-Bracci}). 

The two questions are intertwined: if  the parabolic model dichotomy holds, then from \eqref{step-modello} it immediately follows that the parabolic step dichotomy holds. On the other hand, if there exists a parabolic holomorphic self-map of the ball with nontrivial elliptic canonical model {\sl different from the identity}, then that map  does not satisfy the parabolic step dichotomy.       
In this paper we show that such a map $f$ actually exists, and thus the answer to both Question \ref{Q0} and Question \ref{Q} are negative. This shows that parabolic dynamics displays new and unexpected behaviour in several complex variables.
Next  we describe the geometrical method used to construct such map.
\subsection{Embedding the ball in the bidisc}
Let $\D:=\{\zeta\in\C: |\zeta|<1\}$ and $\Ha:=\{\zeta\in\C: \Re \zeta>0\}$. In what follows we will refer both to $\D\times \D$ and to $\D\times \Ha$ as the bidisc, since they are biholomorphic via the Cayley transform on the second coordinate.
Let $\theta\in \R$. The univalent parabolic\footnote{In the bidisc   parabolic holomorphic self-maps can be defined using the rate of escape of orbits w.r.t. the Kobayashi distance instead of the dilation, see \cite{Arosio-Gumenyuk}.} self-map of $\D\times \Ha$ defined as
$$g(\zeta,y)=(e^{i\theta}\zeta, y+1),$$
is easily seen to be a counterexample to both dichotomies in the bidisc: the canonical model
is given by the function $\pi_1(\zeta,y)=\zeta$, which semi-conjugates $g$ to the elliptic automorphism $z\mapsto e^{i\theta}z$ of $\D$.
 Moreover (if $\theta\not\in 2\pi\Z$) the map $g$ has both points of zero step and points of nonzero step. 
Our strategy to find an analogous self-map in $\B^2$ is to look for a univalent map $\v\colon \B^2\to \D\times \Ha$ whose image  $\Omega:=\v(\B^2)$ is forward invariant and $g$-absorbing\footnote{A  forward invariant domain is $g$-absorbing if every orbit eventually enters it.}. If such $\v$ exists, then it easily follows from the results in \cite{Arosio-Bracci}  that the self-map of $\B^2$ defined by $f:= \v^{-1}\circ g\circ \v$ is parabolic and its canonical model is given by the function $\ell:=\pi_1\circ \v$ which semi-conjugates $f$ to $z\mapsto e^{i\theta}z$, $z\in\D$. 
%Notice also that this argument can also be reversed: if a parabolic univalent self-map $f$ of $\B^2$ contradicts model dichotomy  then it can be necessarily constructed in a similar way, see \cite{Arosio-Bracci}.

For all $r>0$ let $\D_r:=r\D$. For all $R>0$ we  denote by  $S_R:=\{w\in \C: |\Im w|<R\}$, and for all $M\geq 0$ and $R>0$ we set $$S_{R,M}:=\{w\in \C:  | \Im w|< R,\Re w> M\}.$$
The domain $\Omega$ is $g$-absorbing if and only if for all $0<r<1$, $R>0$ there exists $M\geq 0$ such that 
\begin{equation}\label{geometric}
\D_r\times S_{R,M}\subset \Omega.
\end{equation}
Thus we now face  the following geometric problem:
does there exist a domain $\Omega\subset \D\times \mathbb{H}$ which is biholomorphic to $\B^2$, forward invariant by $g(\zeta,y)=(e^{i\theta}\zeta, y+1)$ for some $\theta$, and satisfies \eqref{geometric}?

It would be natural to guess that the  metric structure of the ball does not allow for such an embedding into the bidisc  to exist. This turns out to be   partly true:
being biholomorphic to $\B^2$, the domain  $\Omega$ has to be Gromov hyperbolic (w.r.t. the  Kobayashi distance), while the bidisc is not Gromov hyperbolic due to its product structure.
This gives a surprising quantitative obstruction to the existence of $\Omega$: if the ratio $R/M$ does not go to 0 fast enough compared with $1/|\log (1-r)|$, then $\Omega$ cannot be Gromov hyperbolic. This is the content of our first result.

\begin{theorem}\label{main}
Let $\Omega\subset \D\times \Ha$ be a complete Kobayashi hyperbolic domain with the property that 
 for all $0<r<1$ there exist  $R>0$ and $M\geq 0$ such that $\D_r\times S_{R,M}\subset \Omega$. 
For every $r\in (0,1)$ let
\[
\sigma(r):=\sup\left\{\frac{R}{M}\ \colon \D_r\times S_{R,M}\subset \Omega\right\}.
\]
If \begin{equation}\label{Eq:suff-cond}
\limsup_{r\to 1}\sigma(r)|\log (1-r)|=+\infty,
\end{equation}
 then $(\Omega, k_\Omega)$ is not Gromov hyperbolic. 
\end{theorem} The proof is based on carefully constructing a sequence of $(2,0)$-quasi-geodesics which are not uniformly slim, see Proposition \ref{Prop:quasi-geo}.
Our second result shows that if we allow $\sigma(r)$ to go to 0 fast enough, then it is possible to embed the ball  in the bidisc with the desired properties. 

\begin{theorem}\label{Thm:contro}
The image $\Omega$ of the univalent map $\v\colon \B^2\to \D\times \Ha$ defined by
$$\v(z,w)=\left(\frac{1}{\sqrt{2}}\frac{w}{\sqrt{1-z}}, \frac{1+z}{1-z}   \right)$$ is 
forward invariant under the map $g(\zeta,y)= (e^{i\theta}\zeta, y+1)$ for all $\theta\in \R$, and $g$-absorbing. Moreover $\sigma(r)=\sqrt{\frac{1-r^4}{r^4}}$.
 \end{theorem}

Defining the holomorphic self-map of $\B^2$  as $f:=\v^{-1}\circ g\circ \v$ we have the following corollary, showing that the answer to Questions \ref{Q0} and \ref{Q} is negative.
\begin{corollary}\label{Cor:semimodel-ex}
For all $\theta\in \R$ the parabolic univalent self-map of $\B^2$ defined by
$$f(z,w)=\left(\frac{1+z}{3-z}, \frac{\sqrt 2 e^{i\theta} w}{\sqrt{3-z}}\right)$$
has an elliptic  canonical model given by the map $\ell(z,w)=\frac{1}{\sqrt{2}}\frac{w}{\sqrt{1-z}}$,
which semi-conjugates $f$ to the automorphism $z\mapsto e^{i\theta}z$ on $\D$.
If $\theta\not\in 2\pi\Z$, then  every point with $w=0$ has zero step and every point with $w\neq 0$ has  nonzero step.
\end{corollary}
Every orbit of $f$ converges to the Denjoy--Wolff point $(1,0)$.
Notice that the slice $\{w=0\}$ is invariant by $f$ and in the slice every orbit converges to $(1,0)$ radially, while for every $(z_0,w_0)$ with $w_0\neq 0$ the orbit $(f^n(z_0,w_0))_{n\geq 0}$ converges tangentially to   $(1,0)$.

In fact, the domain $\Omega$ constructed in Theorem~\ref{Thm:contro} is forward invariant under $(\zeta,y)\mapsto (e^{it\theta}\zeta, y+t)$ for all $t\geq 0$. Therefore, the map in the previous corollary gives an example of a parabolic semigroup of holomorphic self-maps of $\B^2$, defined by
$$f_t(z,w)=\left(\frac{t+z(2-t)}{2+t-zt},\frac{\sqrt 2e^{it\theta} w}{\sqrt{2+t-zt}}   \right),\quad t\geq 0,$$
 with elliptic canonical model given by the map $\ell(z,w)=\frac{1}{\sqrt{2}}\frac{w}{\sqrt{1-z}}$ which semi-conjugates $(f_t)$ to the semigroup of automorphisms $(z\mapsto e^{it\theta}z)_{t\geq 0}$ of $\D$.
 
In Proposition~\ref{Prop:semimodel-ex} we generalize the previous results to a larger class of examples.

\medskip

We sincerely thank the anonymous referee for many helpful comments which improve the original manuscript.

\section{A quantitative obstruction to  Gromov hyperbolicity}
In this section we prove Theorem \ref{main}.
We first recall some definitions. For a domain in $\C^q$, we denote by $\kappa_X$ the Kobayashi--Royden metric of $X$ and by $k_X$ the Kobayashi distance of $X$.
\begin{definition}
Let $\Omega\subset \C^q$ be a Kobayashi hyperbolic domain. 
%A {\sl  geodesic} is a map $\gamma$ from an interval $I\subset \R$ to $D$ which is an isometry with respect to the Euclidean distance on $I$ and the Kobayashi distance on $D$, that is for all $s,t\in I$, $$k_D(\gamma(s),\gamma(t))=|t-s|.$$
%If the interval is closed and bounded (resp. $\R_{\geq 0}$,  $\R$) we call $\gamma$ a {\sl geodesic segment}  (resp. {\sl geodesic ray}, {\sl geodesic line}). A geodesic ray $\gamma$ is said to have {\sl endpoint} $\gamma(+\infty)=\xi\in \partial D$ if $\lim_{t\to+\infty}\gamma(t)=\xi$ (and similarly for geodesic lines).
Fix $A\geq1$, $B\geq0$. Let $I\subset \R$ be an interval.
A (non-necessarily continuous) map $\sigma\colon I\to \Omega$ is  an $(A,B)$-\textit{quasi-geodesic } if for every $t,s\geq0$
$$A^{-1}|t-s|-B\leq k_\Omega(\sigma(t),\sigma(s))\leq A|t-s|+B.$$
\end{definition}
\begin{remark}
Fix $A\geq1$, $B\geq0$. Let $[a,b]\subset \R$ be an interval. Let   $\sigma\colon [a,b]\to \Omega$ be a Lipschitz  curve such that $$\ell_\Omega(\sigma;[s,t])\leq Ak_\Omega(\sigma(s),\sigma(t))+B,\quad \forall a\leq s< t\leq b,$$ 
where $\ell_\Omega(\sigma;[s,t])$ denotes the hyperbolic length  ({\sl i.e.},  the distance w.r.t. $k_\Omega$) of the curve $\sigma$ restricted to $[s,t]$.
Then it is easy to see that the curve $\sigma$ reparametrized by hyperbolic arc length is an $(A,B)$-quasi-geodesic.

\end{remark}
\begin{definition}
Let $\Omega\subset \C^q$ be a Kobayashi hyperbolic domain.
The metric space $(\Omega,k_\Omega)$ is Gromov hyperbolic\footnote{This is not the classical definition of Gromov hyperbolicity for length metric spaces, but it is equivalent to the classical one in proper geodesic metric spaces by the so-called Gromov's shadowing lemma (or geodesic stability lemma)} if for all  $A\geq 1,B\geq 0$ there exists a constant $G>0$, such that    every $(A,B)$-quasi-geodesic triangle is $G$-slim, that is, any side is contained in the $G$-neighborhood of the union of the two other sides. 
\end{definition}

\begin{remark}\label{Rem:stimar}
For all $0<r<1$ we have $\kappa_{\D_r}(z;v)=\frac{r|v|}{r^2-|z|^2}.$
If $r> 1/2$, we have that 
$\kappa_{\D_r}(z;v)\leq 2 \kappa_{\D}(z;v)$ if and only if $z\in \overline{\D_{r'}}$, where $$r'(r):=\sqrt{\frac{2r^2-r}{2-r}}.$$
Hence, for every $\zeta, \xi\in \overline{\D_{r'}}$ we have $k_{\D_r}(\zeta, \xi)\leq 2 k_\D(\zeta, \xi)$. 

A straightforward calculation shows that condition \eqref{Eq:suff-cond} is  equivalent to
\begin{equation}\label{Eq:suff-cond-intrinsic}
\limsup_{r\to 1}\sigma(r)k_{\D_r}(s_0,r'(r))=+\infty,
\end{equation}
for any fixed $s_0\in [0,1)$.
\end{remark}

\begin{proposition}\label{Prop:quasi-geo}
Let $\Omega\subset \D\times \C$ be  a complete Kobayashi hyperbolic domain. Suppose there exist $\frac{1}{2}<r<1,R>0$ and $M\geq 0$ such that $\D_r\times S_{R,M}\subset \Omega.$
Fix $c>1$. Then there exists $D=D(c)>0$ such that if $b\geq a\geq M+DR$ and $-r'\leq t_0,  t_1\leq r'$, $t_0\neq t_1$ satisfy
\begin{equation}\label{Eq:cond-for-quasi-geo}
b-a\leq \frac{2R}{c\pi} k_{\D_r}(t_0,t_1),
\end{equation}
then $\gamma:[0,1]\to \Omega$ defined as
\[
\gamma(t):=\left( (t_1-t_0)t+t_0, (b-a)\frac{k_{\D_r}(t_0, (t_1-t_0)t+t_0)}{k_{\D_r}(t_0, t_1)}+a   \right)
\]
is a $(2,0)$-quasi-geodesic joining $\gamma(0)=(t_0, a)$ with $\gamma(1)=(t_1, b)$.
\end{proposition}
\begin{proof}
Fix $c>1$. Arguing as in the proofs of \cite[Prop. 6.8.5 and Prop. 6.7.2]{BCD} one sees that there exists $D>0$ (depending only on $c$) such that if $s\geq DR+M$ then 
\begin{equation}\label{Eq:stri-semi}
\kappa_{S_{R,M}}(s, v)\leq c\kappa_{S_R}(s, v).
\end{equation}
Let $a, b, t_0, t_1$ and $\gamma=(\gamma_1,\gamma_2)$ be as above. Fix $0\leq s\leq t\leq 1$. Then
\[
\ell_{\Omega}(\gamma;[s,t])\leq \ell_{\D_r\times S_{R,M}}(\gamma;[s,t])=\int_s^t \kappa_{\D_r\times S_{R,M}}(\gamma(u); \gamma'(u))du.
\]
Now,
\[
\kappa_{\D_r\times S_{R,M}}(\gamma(u); \gamma'(u))=\max\{\kappa_{\D_r}(\gamma_1(u); \gamma_1'(u)),\kappa_{S_{R,M}}(\gamma_2(u); \gamma_2'(u)) \}.
\]
Assume for the moment that for $u\in [s,t]$ we have
\begin{equation}\label{Eq:max-first}
\max\{\kappa_{\D_r}(\gamma_1(u); \gamma_1'(u)),\kappa_{S_{R,M}}(\gamma_2(u); \gamma_2'(u)) \}=\kappa_{\D_r}(\gamma_1(u); \gamma_1'(u)).
\end{equation}
Then we have, by Remark~\ref{Rem:stimar},
\begin{equation*}
\begin{split}
\ell_{\Omega}(\gamma;[s,t])&\leq \int_s^t\kappa_{\D_r}(\gamma_1(u); \gamma_1'(u))du=k_{\D_r}((t_1-t_0)s+t_0,(t_1-t_0)t+t_0)\\&\leq 2 k_{\D}((t_1-t_0)s+t_0,(t_1-t_0)t+t_0)=2k_\D(\pi_1(\gamma(s)), \pi_1(\gamma(t)))\leq 2k_\Omega(\gamma(s), \gamma(t)),
\end{split}
\end{equation*}
and we are done. Notice that in the last inequality we used $\pi_1(\Omega)\subset \D$.

So we are left to prove \eqref{Eq:max-first}. We have
\begin{equation}\label{Eq:disc-met}
\kappa_{\D_r}(\gamma_1(u); \gamma_1'(u))=\frac{r|t_1-t_0|}{r^2-((t_1-t_0)u+t_0)^2}.
\end{equation}
On the other side, by \eqref{Eq:stri-semi}, we have
\begin{equation}\label{Eq:strip-met}
\begin{split}
\kappa_{S_{R,M}}(\gamma_2(u); \gamma_2'(u))&\leq c\kappa_{S_R}(\gamma_2(u);\gamma_2'(u))=\frac{c\pi}{2R}\gamma_2'(u)\\&=\frac{c\pi}{2R}\frac{b-a}{k_{\D_r}(t_0,t_1)}\frac{\partial}{\partial u}\left(k_{\D_r}\left(t_0, (t_1-t_0)u+t_0\right)  \right)\\&=\frac{c\pi}{2R}\frac{b-a}{k_{\D_r}(t_0,t_1)}|t_1-t_0|\kappa_{\D_r}((t_1-t_0)u+t_0;1)\\&=\frac{c\pi}{2R}\frac{b-a}{k_{\D_r}(t_0,t_1)}\frac{r|t_1-t_0|}{r^2-((t_1-t_0)u+t_0)^2},
\end{split}
\end{equation}
where, in order to compute the penultimate equality we used that, if $t_1\geq t_0$,
\[
\frac{\partial}{\partial u}k_{\D_r}\left(t_0, (t_1-t_0)u+t_0\right)=\frac{\partial}{\partial u}\int_{t_0}^{(t_1-t_0)u+t_0}\kappa_{\D_r}(v;1)dv,
\] and a similar argument if $t_0>t_1$.
Hence, \eqref{Eq:cond-for-quasi-geo}, together with  \eqref{Eq:disc-met} and \eqref{Eq:strip-met},  implies \eqref{Eq:max-first}.
\end{proof}

\begin{proof}[Proof of Theorem~\ref{main}]
Let $r_k\to 1$ be any sequence in $(\frac{1}{2},1)$. Then for each $k$ there exist $R_k>0$ and $M_k\geq 0$ such that $\D_{r_k}\times S_{R_k,M_k}\subset \Omega$.

Fix $c>1$ and let $D>0$ be given by Proposition~\ref{Prop:quasi-geo}. Let $a_k\geq M_k+DR_k$, and let $b_k>a_k$ be such that
\begin{equation}\label{Eq:condi-qg-pf}
b_k-a_k\leq \frac{2R_k}{c\pi} k_{\D_{r_k}}(0,r'_k),
\end{equation}
where $r'_k:=r'(r_k)$, as defined in Remark~\ref{Rem:stimar}.

Since $k_{\D_{r_k}}(0,r'_k)=k_{\D_{r_k}}(0,-r'_k)$, it follows from Proposition~\ref{Prop:quasi-geo} that 
\[
\gamma^-(t)=\left( r_k' (t-1), (b_k-a_k)\frac{k_{\D_{r_k}}(-r_k', r_k'(t-1))}{k_{\D_{r_k}}(-r_k', 0)}+a_k \right)
\]
and
\[
\gamma^+(t)=\left( r_k' (1-t), (b_k-a_k)\frac{k_{\D_{r_k}}(r_k', (1-t) r_k')}{k_{\D_{r_k}}(r_k', 0)}+a_k \right)
\]
are two $(2,0)$-quasi-geodesics such that $\gamma^-$ joins $(-r'_k, a_k)$ to $(0, b_k)$ and $\gamma^+$ joins  $(r'_k, a_k)$ to $(0, b_k)$.  Moreover  $\alpha_k(t)=(r_k' (2t - 1), a_k)$ is a $(2,0)$-quasi-geodesic which joins $(-r_k',a_k)$ with $(r_k',a_k)$.

Assume that $(\Omega, k_\Omega)$ is Gromov hyperbolic. Then there exists $G\geq 0$ 
such that for all $k\geq 0$ the $(2,0)$-quasi-geodesic triangle $(\gamma_k^-,\gamma_k^+,\alpha_k)$ is $G$-slim.
Let $s_0\in [0,1)$ be such that $k_\D(-s_0, 0)=2G$. Let $s_0<p<1$.  Fix $k_0\in \N$ such that  $r'_k\geq p$ for every $k\geq k_0$. Let $k\geq k_0$. If $t\in [0, 1-\frac{s_0}{r'_k}]$, then since $\pi_1(\Omega)\subset \D$ we have 
\[
k_\Omega(\gamma^-(t), \gamma^+([0,1]))\geq k_\D(\gamma_1^-(t), 0)=k_\D( r_k' (t-1), 0)\geq k_\D(0, -s_0)=2G.
\]
Therefore, since $(\gamma_k^-,\gamma_k^+,\alpha_k)$ is $G$-slim, for all $k\geq k_0$ and $t\in [0, 1-\frac{s_0}{r'_k}]$ we have
\[
k_\Omega(\gamma_k^-(t), \alpha_k)\leq G.
\]
Set
\[
q_k:=\gamma^-_k\left(1-\frac{s_0}{r'_k}\right)=\left(-s_0, (b_k-a_k)\frac{k_{\D_{r_k}}(-r_k', -s_0)}{k_{\D_{r_k}}(-r_k', 0)}+a_k \right).
\]
Notice that $k_{\D_{r_k}}(-r_k', -s_0)=k_{\D_{r_k}}(r_k', s_0)$.
Since $\pi_2(\Omega)\subset\mathbb{H}$,
\[
k_\Omega(q_k, \alpha_k([0,1]))\geq k_{\Ha}( (b_k-a_k)\frac{k_{\D_{r_k}}(r_k', s_0)}{k_{\D_{r_k}}(r_k', 0)}+a_k , a_k)=\frac{1}{2}\log\left(  \frac{(b_k-a_k)}{a_k}\frac{k_{\D_{r_k}}(r_k', s_0)}{k_{\D_{r_k}}(r_k', 0)}+1\right).
\]
Therefore, if $L>0$ is such that $\frac{1}{2}\log (L+1)>G$, we have necessarily that for all $k\geq k_0$,
\begin{equation}\label{Eq:stay-bd}
\frac{(b_k-a_k)}{a_k}\frac{k_{\D_{r_k}}(r_k', s_0)}{k_{\D_{r_k}}(r_k', 0)}\leq L.
\end{equation}
Now, let $a_k=M_k+DR_k$ and $b_k=a_k+\frac{2R_k}{c\pi} k_{\D_{r_k}}(0,r'_k)$. Note that \eqref{Eq:condi-qg-pf} is satisfied and thus \eqref{Eq:stay-bd} holds, that is, 
\[
\frac{2R_kk_{\D_{r_k}}(r_k', s_0)}{c\pi a_k}\leq L,
\] 
or, equivalently,
\[
\frac{M_k}{R_k k_{\D_{r_k}}(r_k', s_0)}\geq \frac{2}{c\pi L}-\frac{D}{k_{\D_{r_k}}(r_k', s_0)}.
\]
Hence
\[
\limsup_{k\to \infty}\frac{R_k k_{\D_{r_k}}(r_k', s_0)}{M_k}\leq \frac{c\pi L}{2}<+\infty.
\]
Taking into account Remark~\ref{Rem:stimar}, we have that
$$\limsup_{r\to 1}\sigma(r)|\log(1-r)|\leq\frac{c\pi L}{2}<+\infty.$$
\end{proof}

\begin{corollary}\label{forintrinsic}
Let $\Omega\subset \D\times \Ha$ be a complete Kobayashi hyperbolic domain. Assume that there exists $C>0$ such that 
for all $0<r<1$ there exists  $M\geq 0$ such that 
$$ \D_r\times S_{CM,M}\subset \Omega.$$  
Then $(\Omega,k_\Omega)$ is not Gromov hyperbolic.
\end{corollary}
\begin{proof}
It immediately follows from  Theorem \ref{main}  since $\sigma(r)\geq  C$ for all $0<r<1$.
\end{proof}

 For $C>0$ and $M\geq 0$ we denote by $ V_{C,M}\subset \Ha$ the angular subset defined by
 $$ V_{C,M}:=\{w\in \C: \Re w>M, |\Im w|<C\Re w\}.$$
 Let  $\Phi\colon \D\times \Ha\to \D\times \D$ be the biholomorphism given by the inverse of the Cayley transform on the second coordinate, that is
\begin{equation}\label{Eq:formulaPhi}
 \Phi(z,w):=(\Phi_1(z), \Phi_2(w)):=\left(z,\frac{w-1}{w+1}\right).
\end{equation} 
 Let $\tilde V_{C,M}:=\Phi_2(V_{C,M})\subset \D$.
The open set $\tilde V_{C,M}$ is a sector-shaped domain of $\D$ whose closure intersects $\partial \D$ only at $1$, with angle $\theta=2\arctan C$. Using the map $\Phi$ we can obtain an obstruction to Gromov hyperbolicity for domains $\tilde \Omega\subset \D\times \D$   whose closure contains in the boundary the disc $\{1\}\times \overline \D$, in terms of the products of the form $ \D_r\times \tilde V_{C,M}$ contained in $\tilde \Omega$.

 \begin{corollary}\label{sectors}
  Let  $\tilde \Omega$ be a  complete  Kobayashi hyperbolic domain contained in the bidisc $\D\times \D$.
  Assume that there exists $C>0$ such that 
for all $0<r<1$ there exists  $M\geq 0$ such that 
$$ \D_r\times  \tilde V_{C,M}\subset\tilde  \Omega.$$  
Then $(\tilde \Omega,k_{\tilde \Omega})$ is not Gromov hyperbolic.
 \end{corollary}
 \begin{proof}
 It immediately follows from Corollary \ref{forintrinsic} since $S_{CM,M} \subset V_{C,M}$.
 \end{proof}
 
% 
%The following ``intrinsic'' result is now an easy consequence of Corollary \ref{forintrinsic}.
% \begin{corollary}
% Let $D_1,  D_2\subset\C$ be proper simply connected domains. Let $z_0\in D_1$ and let $\gamma:[0,+\infty)\to D_2$ be a geodesic ray (with respect to the hyperbolic distance in $D_2$). Let $B(z_0, r)$ denote the hyperbolic disc of center $z_0$ and radius $r>0$ in $D_1$. If $\Omega\subset D_1\times D_2$ is a domain such that
% there exists $C_0>0$ so that for all $r\in (0,+\infty)$  there exists   $T\geq 0$ such that 
%\[
%B(z_0, r)\times \{w\in D_2: k_{D_2}(w, \gamma([T,+\infty))<C_0\}\subset \Omega,
%\]
%then $(\Omega, k_\Omega)$ is not Gromov hyperbolic.
% \end{corollary}
%\begin{proof}
%Let the biholomorphism $F:\D \times \Ha\to D_1\times D_2$ be defined as $F(z,w)=(f(z), g(w))$, where $f:\D\to D_1$ is a Riemann map such that $f(0)=z_0$ and $g:\Ha\to D_2$ is a Riemann map such that $g([1,+\infty])=\gamma([0,+\infty))$. Let $\tilde \Omega:=F^{-1}(\Omega)$. By  \cite[Lemma 6.2.3]{BCD}, for all $T\geq 0$ there exist $C=C(C_0)>0$ and $M>0$ such that
%\[
%\{w\in \C: |w|>M, |\Im w|<C\Re w\}\subset g^{-1}( \{w\in D_2: k_{D_2}(w, \gamma([T,+\infty))<C_0\}).
%\]
%This implies that for all $r\in (0,1)$ there exists $M>0$ such that  $\D_r\times S_{CM,M}\subset\tilde \Omega$.
%By Corollary \ref{forintrinsic} the domain $\tilde\Omega$ is not Gromov hyperbolic, and thus $\Omega$ is also not Gromov hyperbolic.
%\end{proof}

\section{Construction of the counterexample}\label{counterexample}
\begin{proof}[Proof of Theorem ~\ref{Thm:contro}]
We construct  a domain $ \Omega$ in $\C^2$ such that
\begin{itemize}
\item[i)] $ \Omega$ is bihomomorphic to the ball $\B^2$ and $ \Omega\subset \D\times  \Ha$,
\item[ii)] for all $0<r<1$ and  $R>0$ there exists $M\geq 0$ such that 
$$
\D_r\times S_{R,M}\subset \Omega.
$$
\item[iii)] $ \Omega$ is forward invariant by the map $g(\zeta,y)= (e^{i\theta}\zeta, y+1)$ for all $\theta\in \R$.
\end{itemize}

%We  construct the domain $\tilde \Omega:=F(\B^2)$ as the image of the univalent map 
 %$F:\B^2\to \C^2$  defined by
 Consider the biholomorphism
$\v\colon \D\times \C\to \C\times  \Ha$ defined by
\[
\v(z,w)=\left(\frac{1}{\sqrt{2}}\frac{w}{\sqrt{1-z}}, \frac{1+z}{1-z}   \right).
\]
We will show that $ \Omega:=\v(\B^2)$ satisfies i),ii),iii). Since for $(z,w)\in \B^2$
\[
\frac{1}{2}\frac{|w|^2}{|1-z|}<\frac{1}{2}\frac{1-|z|^2}{|1-z|}<\frac{1-|z|}{|1-z|}\leq 1,
\]
it follows that $ \Omega\subset \D\times  \Ha$, and thus i) is satisfied. In order to prove ii) and iii) we will first show that  $ \Omega$ is described by the following inequality:

\begin{equation}\label{eqomega}
 \Omega:=\left\{(\zeta,y)\in \C\times  \Ha :\quad |\zeta|^2<\frac{\Re y}{|1+y|}\right\}.
\end{equation}
Indeed, the inverse $\psi:=\v^{-1}\colon \C\times  \Ha\to  \D\times \C$ is given by
$$\psi(\zeta,y)=\left(\frac{y-1}{y+1},\frac{2\zeta}{\sqrt{1+y}}\right),$$
and since  $\Omega=\psi^{-1}(\B^2)$,
$$ \Omega:=\left\{(\zeta,y)\in \C\times  \Ha :\quad \frac{|y-1|^2}{|y+1|^2}+\frac{4|\zeta|^2}{|1+y|}<1 \right\},$$
which is easily seen to be equivalent to \eqref{eqomega}.

Now, fix $r\in (0,1)$ and $R>0$. 
By \eqref{eqomega}, the inclusion $\D_r\times S_{R,M}\subset \Omega$ is equivalent to the existence of $M\geq 0$ such that
\[
Q(M):= \frac{M}{\sqrt{(1+M)^2+R^2}}\geq r^2.
\]
Since  $\lim_{M\to \infty}Q(M)=1$ while $r^2<1$ the latter inequality holds for all $M\geq 0$ large enough. Thus, condition ii) holds.

To prove condition iii) we have to show that if $(\zeta, y)\in \Omega$, then $(e^{i\theta}\zeta, y+1)\in  \Omega$, which is equivalent to
\begin{equation}\label{Eq:condit-inv}
|\zeta|^2<\frac{\Re y+1}{|y+2|}.
\end{equation}
Now, for $\Im y$ fixed, the function
\[
\Re y\mapsto \frac{|1+y|}{\Re y}=\sqrt{1+\frac{2}{\Re y}+\frac{1+(\Im y)^2}{(\Re y)^2}}
\]
is decreasing on $(0,+\infty)$. Hence,
\[
|\zeta|^2<\frac{\Re y}{|1+y|}<\frac{\Re y+1}{|y+2|},
\]
and \eqref{Eq:condit-inv} holds.

The formula for $\sigma(r)$ will be established in Proposition~\ref{Prop-comp}.
\end{proof}
Let $\Omega$ be given by \eqref{eqomega} and $\Phi$ be the map given by \eqref{Eq:formulaPhi}. 
The domain $\tilde \Omega:=\Phi(\Omega)\subset \D\times \D$ is biholomorphic to the ball $\B^2$ and contains in the boundary the disc $\{1\}\times \overline \D$. A direct computation shows that actually
\[
\tilde \Omega=\{(\zeta,y)\in\C^2: |y|^2+2|\zeta|^2|1-y|<1\}.
\]

We now show that for all $r<1$ there exists $C>0,M\geq 0$ such that 
$\D_r\times \tilde V_{C,M}\subset \tilde\Omega$, and that, in accordance with Corollary~\ref{sectors}, $C\to 0$ as $r\to 1$.
\begin{proposition}\label{stolz}
For all $0<r<1$ there exists $M\geq 0$ such that $$\D_r\times  \tilde V_{C,M}\subset\tilde \Omega$$ if and only if
$C<\sqrt{\frac{1-r^4}{r^4}}.$
\end{proposition}
\begin{proof}
For every $0<r<1$ we look for $C>0$ and  $M\geq 0$ such that  $\D_r\times  V_{C,M}\subset \Omega$. 
By \eqref{eqomega} this means, setting $y=t + iat$, that for all $a\in (-C,C)$ and $t> M$ and for all $|\zeta|< r$ we have
$$\frac{t}{\sqrt{(1+t)^2+a^2t^2}}>|\zeta|^2,$$ that is
$$\frac{1}{|\zeta|^4}-1-a^2>\frac{1}{t^2}+\frac{2}{t}.$$
This is possible if and only if $1+C^2<\frac{1}{r^4}$, or equivalently $C<\sqrt{\frac{1-r^4}{r^4}}.$
\end{proof}
We end this section computing $\sigma$ for the domain $\Omega$.
\begin{proposition} \label{Prop-comp}
For all $r\in (0,1)$ we have that
$$\sigma(r)=\sqrt{\frac{1-r^4}{r^4}}.$$
\end{proposition}
\begin{proof}
Fix $\zeta_0\in \D$ and denote $y=x+iu$. Then $(\zeta_0,y)\in \Omega\cap \{\zeta=\zeta_0\}$
if and only if $x>0$ and  $$x^2(1-\frac{1}{|\zeta|^4})+u^2+2x+1<0.$$ Setting $\nu:=\frac{1}{|\zeta|^4}-1$ we get 
$$\frac{(x-\frac{1}{\nu})^2}{\frac{\nu+1}{\nu^2}}-\frac{u^2}{\frac{\nu+1}{\nu}}>1,$$
whose boundary is the right branch of a hyperbola whose axes have angular coefficient $\pm \sqrt{\frac{1-|\zeta^4|}{|\zeta|^4}}$. It follows that if for a given $0<r<1$ we have that there exists $R>0,M\geq 0$ such that 
$\D_r\times S_{R,M}\subset \Omega$, then we must have $\frac{R}{M}\leq  \sqrt{\frac{1-r^4}{r^4}}.$
On the other side, by the proof of Proposition \ref{stolz} we have 
$\sigma(r)\geq \sqrt{\frac{1-r^4}{r^4}}$, and we are done.
\end{proof}

\begin{proof}[Proof of Corollary~\ref{Cor:semimodel-ex}]
Theorem~\ref{Thm:contro} shows that the triple $(\D\times \C, \varphi, g)$ is the ``model'' (in the sense of \cite[Definition~3.3]{Arosio-Bracci}) of the univalent self-map $f$ of $\B^2$. 
It follows by \cite[Theorem~4.10]{Arosio-Bracci} that $f$ has an elliptic  canonical model given by the map $\ell(z,w)=\frac{1}{\sqrt{2}}\frac{w}{\sqrt{1-z}}$.
The statement on the steps follows immediately from  \eqref{step-modello}. 

Since by \cite[Lemma 4.12]{Arosio-Bracci} a hyperbolic map cannot have an elliptic  canonical Kobayashi hyperbolic semi-model, it follows that the map $f$ is parabolic. Alternatively, one can obtain the same conclusion by noticing that the slice $\{w=0\}$ of $\B^2$ is $f$-invariant and $f(z,0)=(\frac{1+z}{3-z},0)$. Taking into account that $\D\ni z\mapsto \frac{1+z}{3-z}$ is a parabolic self-map of $\D$ it follows at once that $f$ is a parabolic self-map of $\B^2$.
\end{proof}

\section{A class of examples}

The aim of this section is to provide a larger class of examples of parabolic univalent self-maps of $\B^2$ with an elliptic semi-model. In fact, our construction will provide examples given by (continuous) semigroups of holomorphic self-maps of $\B^2$. In order to do this, we need some preliminary result. 

\begin{definition}
A   curve $\gamma\colon [0,1)\to \D$ converges to $1$ {\sl radially} (or orthogonally) if
$\lim_{t\to 1}\gamma(t)=1$ and if $$\frac{\gamma(t)-1}{\|\gamma(t)-1\|}\to -1,$$
or, equivalently, if $k_\D(\gamma(t),I)\to 0$ as $t\to 1^-$, where $I$ denotes the 
interval $[0,1)\subset \R$.
\end{definition}
\begin{definition}
A domain $A\subset \C$ is {\sl starlike at infinity} if for all $z\in A$ and $t\geq 0$ we have $z+t\in A.$
\end{definition}
Let $h:\D\to \C$ be a univalent function such that $h(\D)\subset\Ha$ is starlike at infinity and contains a point $t_0\in \R_{> 0}$. We consider the biholomorphism $\v_h:\D\times \C\to \C\times h(\D)$ given by
\begin{equation}\label{Eq:defF}
\v_h(z,w)=\left(\frac{1}{\sqrt{2}}\frac{w}{\sqrt{1-z}}, h(z)  \right).
\end{equation}
Clearly if we take as $h$ the Cayley transform $\mathscr{C}: \D\ni z   \mapsto \frac{1+z}{1-z}$ we obtain the map $\v$ employed in the previous section.

We let
\[
\Omega_h:= \v_h(\B^2).
\]
 Then $(\Omega_h, k_{\Omega_h})$ is Gromov hyperbolic and, by construction, $\Omega_h\subset \D\times \Ha$.
 
Note that
\begin{equation}\label{Eq:Omega-h}
\Omega_h:=\{(\zeta, y)\in \C\times h(\D): |h^{-1}(y)|^2+2|\zeta|^2|1-h^{-1}(y)|<1\}.
\end{equation}

\begin{proposition}\label{Prop:tech-for-h} The following hold:
\begin{enumerate}
\item For all $0<r<1$ there exist  $R>0$ and $M\geq 0$ such that 
\begin{equation}\label{Eq:cond-inscr-weak}
 \D_r\times S_{R,M}\subset \Omega_h
\end{equation}
if and only if $h^{-1}(t)\to 1$ radially as $t\to +\infty$.
\item For all $0<r<1$ and $R>0$  there exists $M>0$ such that 
\begin{equation}\label{Eq:cond-inscr-strong}
\D_r\times S_{R,M}\subset \Omega_h
\end{equation}
if and only if $h^{-1}(t)\to 1$ radially and $\delta_{h(\D)}(t)\to+\infty$ as $t\to+\infty$. Here
$\delta_{h(\D)}(t)$ denotes the Euclidean distance between $t$ and the boundary of $h(\D)$.
\end{enumerate}
\end{proposition}
\begin{proof} 
(1) Consider the map $\hat{h}:\Ha\to \Ha$ given by $\hat{h}(z)=h(\frac{z-1}{z+1})$, that is $\hat{h}=h\circ \mathscr{C}^{-1}.$
Then  by \eqref{Eq:Omega-h}, 
\begin{equation}\label{Eq:Omega-tilde-h}
\Omega_h=\left\{(\zeta,y)\in \C\times h(\D):\quad|\zeta|^2 < \frac{\Re (\hat{h}^{-1}(y))}{|1+\hat{h}^{-1}(y)|}\right\}.
\end{equation}
 Given $r\in (0,1)$,  $R>0$ and $M\geq 0$  we have that $\D_r\times S_{R,M}\subset \Omega_h$  if and only if
\begin{equation}\label{check}
S_{R,M}\subset h(\D)= \hat{h}(\mathbb H)
\end{equation}
and 
\begin{equation}\label{check2}
r^2 \leq \frac{\Re(\hat{h}^{-1}(y))}{|1+\hat{h}^{-1}(y)|}, \quad \forall y \in S_{R,M}.
\end{equation}
Assume  that for all $r\in (0,1)$ there exist  $R>0$ and $M\geq 0$ such that  $ \D_r\times S_{R,M}\subset \Omega_h$. Fix $r\in(0,1)$. 
Then for all $t>M$ we have that $t\in \hat{h}(\mathbb H)$ and
 \begin{equation*}
r^2\leq \frac{\Re (\hat{h}^{-1}(t))}{|1+\hat{h}^{-1}(t)|}.
\end{equation*}
This means that
\[
\lim_{t\to +\infty}\frac{\Re (\hat{h}^{-1}(t))}{|1+\hat{h}^{-1}(t)|}=1,
\]
which implies $\hat{h}^{-1}(t)\to \infty$ and $\Arg (\hat{h}^{-1}(t))\to 0$ as $t\to +\infty$. Applying the Cayley transform we have that $h^{-1}(t)\to 1$ radially  as $t\to +\infty$.

Conversely, assume  $h^{-1}(t)\to 1$ radially  as $t\to +\infty$.
Now, fix $t_0\geq 0$ such that $t_0\in h(\D)$ and let  $R_0>0$ be such that $t_0+iu\in h(\D)$ for $u\in [-R_0,R_0]$. Since $h(\D)$ is starlike at infinity, it follows that $t+iu\in h(\D)$ for all $t>t_0$ and $u\in [-R_0,R_0]$, and there exists a constant $C>0$, depending only on $t_0$ and $R_0$ such that $\delta_{h(\D)}(t+iu)\geq C$ for all $t\geq t_0$ and $u\in [-R_0,R_0]$.  Hence, for every $t\geq t_0$ and $u\in [-R_0,R_0]$,
\[
k_{\Ha}(\hat{h}^{-1}(t), \hat{h}^{-1}(t+iu))=k_{h(\D)}(t,t+iu)\leq \ell_{h(\D)}(\gamma;[0,1]),
\]
where $\gamma(s)=t+isu$. By the estimates of the hyperbolic metric (see, {\sl e.g.}, \cite[Thm. 5.2.1]{BCD}), 
\begin{equation}\label{Eq:stima-ell-dist}
\ell_{h(\D)}(\gamma;[0,1])=\int_0^1 \kappa_{h(\D)}(\gamma(s);\gamma'(s))ds\leq \int_0^1 \frac{|u| ds}{\delta_{h(\D)}(t+isu)}\leq \frac{|u|}{C}.
\end{equation}
Let $\epsilon>0$ be such that $R_1:=C\epsilon\leq R_0$. Hence, by \eqref{Eq:stima-ell-dist},
\begin{equation}\label{Eq:epsi-h}
k_{\Ha}(\hat{h}^{-1}(t), \hat{h}^{-1}(t+iu)) \leq \epsilon, \quad \forall t\geq  t_0, |u|\leq R_1.
\end{equation}
By the triangle inequality it follows that there exists $t_1\geq t_0$ such that 
\[
k_{\Ha}(\hat{h}^{-1}(t+iu), [1,+\infty))<2\epsilon, \quad \forall t> t_1, |u|\leq R_1.
\]
By \cite[Lemma 6.2.3]{BCD}, this imples that for all $\epsilon>0$ there exist $M\geq 0$ and $R>0$ such that 
\[
\frac{|\Im \hat{h}^{-1}(t+iu)|}{\Re \hat{h}^{-1}(t+iu)}<\epsilon, \quad \forall t>M, |u|\leq R.
\]

Taking into account that, by the Denjoy-Wolff Theorem, the real part of $\hat{h}^{-1}(t+iu)$ converges uniformly, for $|u|\leq R$, to $+\infty$ as $t\to +\infty$ (being the orbit of the semigroup $\phi_t(w):=\hat{h}^{-1}(\hat{h}(w)+t)$ of $\Ha$), the previous equation  implies that, given $r\in (0,1)$, we can find $M\geq 0$ and $R>0$ such that  \eqref{check2} is satisfied for all $y\in S_{R,M}$, and hence $ \D_r\times S_{R,M}\subset \Omega_h$.

(2) The argument follows the previous ideas, so we just sketch it. If \eqref{Eq:cond-inscr-strong} holds, then clearly by \eqref{check} we have $\delta_{h(\D)}(t)\to +\infty$ as $t\to+\infty$. 

Conversely, if $\delta_{h(\D)}(t)\to +\infty$ as $t\to+\infty$, it  follows that for every $R>0$ there exists $t_0\geq 0$ such that $t+iu\in h(\D)$ for all $t> t_0$ and $u\in [-R,R]$. Moreover,  we have $\delta_{h(\D)}(t+iu)\to +\infty$ as $t\to +\infty$ uniformly in $u\in [-R,R]$. By \eqref{Eq:stima-ell-dist} this means that $k_{\Ha}(\hat{h}^{-1}(t), \hat{h}^{-1}(t+iu))\to 0$ as $t\to +\infty$ uniformly in $u$. Since by assumption $k_{\Ha}(\hat{h}^{-1}(t), [1,+\infty))\to 0$, it follows
by \cite[Lemma~6.2.3]{BCD} that for every $0<r<1$, $R>0$ there exists  $M\geq 0$ such that  \eqref{check2} is satisfied  for all $y\in S_{R,M}$, and hence $ \D_r\times S_{R,M}\subset \Omega_h$.
\end{proof}

\begin{lemma}\label{Lem:inside}
Let $A\subsetneq\C$ be a starlike at infinity domain which is symmetric  with respect to the real axis ({\sl i.e.}, $z\in A$ if and only if $\overline{z}\in A$). Then there exists a biholomorphism $h:\D\to A$  satisfying $h(r)\in \R$ for all $r\in(-1,1)$ and $\lim_{r\to 1}h(r)=+\infty$. Moreover for all $\theta\in\R$ and $t>0$ the domain $\Omega_h$ is invariant under the map $$g_t(\zeta, y)= (e^{it\theta} \zeta, y+t).$$ 
\end{lemma}
\begin{proof} Since $A$ is a simply connected domain different from $\C$ and symmetric with respect to $\R$, it follows (see, {\sl e.g.}, \cite[Prop.~6.1.3]{BCD}) that $\R\cap A$ is a geodesic for $k_A$. Hence, one can choose a Riemann map $h:\D \to A$ which maps the geodesic $(-1,1)$ of $\D$ (w.r.t. $k_\D$) onto $A\cap \R$ and such that  $\lim_{r\to 1}h(r)=+\infty$.

Now, let $\hat{h}:\Ha\to h(\D)$ be given by $\hat{h}(z)=h(\frac{z-1}{z+1})$. Note that, by construction, $\hat{h}(r)\in\R$ for all $r>0$ and  $\lim_{r\to +\infty} \hat{h}(r)=+\infty$.  Notice also that  $\lim_{r\to +\infty} \hat{h}^{-1}(r)=+\infty$.

Since $h(\D)$ is starlike at infinity, then $y+t\in h(\D)$ for every $t\geq 0$ and $y\in h(\D)$. Hence, by \eqref{Eq:Omega-tilde-h}, it is enough to show that the function
\begin{equation*}
\begin{split}
[0,+\infty)\ni t\mapsto L(t)&:=\frac{|1+\hat{h}^{-1}(y+t)|}{\Re (\hat{h}^{-1}(y+t))}\\&=\sqrt{1+\frac{2}{\Re \hat{h}^{-1}(y+t)}+\frac{1}{(\Re \hat{h}^{-1}(y+t))^2}+\frac{(\Im \hat{h}^{-1}(y+t))^2}{(\Re \hat{h}^{-1}(y+t))^2}}
\end{split}
\end{equation*}
is non-increasing. 

To this aim, let $\phi_t(z):=\hat{h}^{-1}(\hat{h}(z)+t)$, for $z\in \Ha$ and $t\geq 0$. Thus, $(\phi_t)_{t\geq 0}$ defines a continuous semigroup of holomorphic self-maps of $\Ha$. Since $\lim_{t\to +\infty}\phi_t(1)=\lim_{t\to +\infty} \hat{h}^{-1}(\hat{h}(1)+t)=+\infty$, the Denjoy-Wolff point of $(\phi_t)_{t\geq 0}$ is $\infty$. By the Denjoy-Wolff theorem for holomorphic self-maps of $\Ha$ (see, {\sl e.g.}, \cite[Theorem~1.7.8]{BCD}), the function $[0,+\infty)\ni t \mapsto \Re \phi_t(z)$ is non-decreasing for all fixed $z\in \Ha$, that is, $[0,+\infty)\ni t\mapsto \Re  \hat{h}^{-1}(y+t)$ is non-decreasing for all fixed $y\in h(\D)$. 

Therefore, in order to show that the function $L$ is non-increasing, we are left to show that $[0,+\infty)\ni t\mapsto\frac{|\Im \hat{h}^{-1}(y+t))|}{\Re \hat{h}^{-1}(y+t)}$ is non-increasing for all fixed $y\in h(\D)$. 

Since for all $\lambda>0$ the map $z\mapsto \lambda z$ is an automorphism of $\Ha$ it easily follows (see also \cite[Lemma~6.2.3]{BCD}) that for all $z\in \Ha$,
$$
k_{\Ha}(z,(0+\infty))=k_{\Ha}(1,e^{i\arctan \frac{|\Im z|}{\Re z} }).
$$
Fix $y\in h(\D)$, and let $0\leq t_0\leq t_1$. Then, taking into account that $(0,+\infty)\ni u\mapsto k_{\Ha}(1,e^{i\arctan u })$ is increasing, we have that
$$\frac{|\Im \hat{h}^{-1}(y+t_1))|}{\Re \hat{h}^{-1}(y+t_1)}\leq \frac{|\Im \hat{h}^{-1}(y+t_0))|}{\Re \hat{h}^{-1}(y+t_0)}$$ is equivalent to
\[
k_{\Ha}(\hat{h}^{-1}(y+t_1), (0,+\infty))\leq k_{\Ha}(\hat{h}^{-1}(y+t_0), (0,+\infty)).
\]
Since $\hat{h}$ is an isometry for the hyperbolic distance and maps $(0,+\infty)$ onto $\R\cap A$, the previous inequality is equivalent to
\[
k_{\hat{h}(\Ha)}(y+t_1, A\cap \R)\leq k_{\hat{h}(\Ha)}(y+t_0, A\cap \R).
\]
Let $x_{t_0}\in A\cap \R$ so that $k_{\hat{h}(\Ha)}(y+t_0, A\cap \R)=k_{\hat{h}(\Ha)}(y+t_0, x_{t_0})$. Taking into account that $w\mapsto w+(t_1-t_0)$ is a holomorphic self-map of $\hat{h}(\Ha)$, thus non-expanding w.r.t. the hyperbolic distance, we have
\begin{equation*}
\begin{split}
k_{\hat{h}(\Ha)}(y+t_1, A\cap \R)&\leq k_{\hat{h}(\Ha)}(y+t_1, x_{t_0}+t_1-t_0)\leq k_{\hat{h}(\Ha)}(y+t_0, x_{t_0})\\&=k_{\hat{h}(\Ha)}(y+t_0, A\cap \R),
\end{split}
\end{equation*}
and we are done.
\end{proof}

Assume now that  if $h(\D)$ is symmetric with respect to the real axis, $h((0,+\infty))=h(\D)\cap\R$, $\lim_{r\to 1}h(r)=+\infty$, and  $\delta_{h(\D)}(t)\to+\infty$ as $t\to+\infty$.
By the previous lemma  it follows that $$f_t:= \varphi_h^{-1}\circ g_t \circ \varphi_h,\quad t\geq 0,$$ (where the map $\varphi_h$ is defined in \eqref{Eq:defF}) defines a  continuous semigroup of holomorphic self-maps of $\B^2$.  Such maps provide a class of examples which generalizes  Corollary~\ref{Cor:semimodel-ex}:

\begin{proposition}\label{Prop:semimodel-ex} Let $h:\D\to \Ha$ be a univalent map such that $h(\D)$ is symmetric with respect to the real axis, $h((-1,1))\subseteq \R$ and $\lim_{r\to 1}h(r)=+\infty$. Assume also that $\delta_{h(\D)}(t)\to+\infty$ as $t\to+\infty$. Let $\theta\in \R$ and for  $t\geq 0$ let
$$f_t(z,w)=\left(h^{-1}(h(z)+t), \frac{ e^{it \theta} w\sqrt{1-h^{-1}(h(z)+t)}}{\sqrt{1-z}}\right).$$
Then $(f_t)$ is a parabolic  (continuous) semigroup of holomorphic self-maps of $\B^2$.  It
has an elliptic  canonical model given by the map $\ell(z,w)=\frac{1}{\sqrt{2}}\frac{w}{\sqrt{1-z}}$,
which semi-conjugates $f_t$ to the group of elliptic automorphisms $z\mapsto e^{it \theta}z$ on $\D$.
In particular, if $\theta\not\in 2\pi\Z$, then  for every point with $w=0$ the map $f_1$ has zero step and for every point with $w\neq 0$ it has  nonzero step.
\end{proposition}
\begin{proof}
By (2) of Proposition~\ref{Prop:tech-for-h}, the triple $(\D\times\C, \varphi_h, g_t)$ is the ``model'' (in the sense of \cite[Definition~6.7]{Arosio-Bracci}) of the semigroup $(f_t)$. Hence, by \cite[Theorem~6.10]{Arosio-Bracci}, $(\D, \ell, z\mapsto e^{it\theta}z)$ is the canonical Kobayashi hyperbolic semi-model of the semigroup $(f_t)$.

The statement on the steps follows immediately from  \eqref{step-modello}. 

Since by \cite[Lemma 4.12]{Arosio-Bracci} a hyperbolic semigroup cannot have an elliptic  canonical Kobayashi hyperbolic semi-model, it follows that $(f_t)$ is parabolic.
Alternatively, one can obtain the same conclusion directly by noticing that the slice $\{w=0\}$ of $\B^2$ is invariant and that  $f_t(\cdot,0)=(h^{-1}(h(\cdot)+t),0)$, $t\geq 0$. Hence, taking into account that $\delta_{h(\D)}(t)\to+\infty$ as $t\to+\infty$, it follows that $(\C, h, z\mapsto z+t)$ is the  model of  the semigroup of holomorphic self-maps of $\D$ defined by $\D\ni z\mapsto h^{-1}(h(z)+t)$, $t\geq 0$, which is thus a parabolic semigroup of zero hyperbolic step (see, {\sl e.g.}, \cite[Thm. 9.3.5]{BCD}).
 \end{proof}

\end{document}